\documentclass[11pt,reqno]{article}

\usepackage{amsmath,amsthm,amsfonts,amssymb,latexsym,mathrsfs,url,enumerate}
\usepackage{graphicx,ifpdf,epstopdf}
\usepackage{comment}
\newtheorem{thm}{Theorem}[section]
\newtheorem{lem}[thm]{Lemma}

\newtheorem{prob}[thm]{Problem}
\newtheorem{que}[thm]{Question}
\newtheorem{conj}[thm]{Conjecture}
\newtheorem{claim}[thm]{Claim}

\theoremstyle{definition}

\numberwithin{equation}{section}

\def\qedB{{\hfill\enspace\vrule height8pt depth0pt width8pt}}

\let\svthefootnote\thefootnote
\newcommand\blankfootnote[1]{%
	\let\thefootnote\relax\footnotetext{#1}%
	\let\thefootnote\svthefootnote%
}

\begin{document}

\title{Optimal bisections of directed graphs}

\author{Guanwu Liu
~~~~~
Jie Ma
~~~~~
Chunlei Zu
}

\date{}

\maketitle\

\begin{abstract}
In this paper, motivated by a problem of Scott \cite{Scott2005} and a conjecture of Lee, Loh and Sudakov \cite{Lee2014} we consider bisections of directed graphs.
We prove that every directed graph with $m$ arcs and minimum semidegree at least $d$ admits a bisection in which at least $\left(\frac{d}{2(2d+1)}+o(1)\right)m$ arcs cross in each direction.
This provides an optimal bound as well as a positive answer to a question of Hou and Wu \cite{Hou2017} in a stronger form.
\end{abstract}

\blankfootnote{\noindent School of Mathematical Sciences, University of Science and Technology of China, Hefei, Anhui 230026, China.
Research supported by the National Key R and D Program of China 2020YFA0713100,
National Natural Science Foundation of China grant 12125106, and Anhui Initiative in Quantum Information Technologies grant AHY150200.}

\section{Introduction}
As one of the most fundamental problems in the common field of combinatorics and theoretical computer science,
graph partitioning problems have been studied extensively in the literature.
A well-known example is the Max-Cut problem, which asks to find a vertex-partition $V_1\cup V_2$ of a given graph with $m$ edges that maximizes the number of edges between $V_1$ and $V_2$.
A uniformly random partition shows that this number is at least $m/2$.
This bound was improved to $m/2+(\sqrt{2m+1/4}-1/2)/4$ by a prominent result of Edwards \cite{Edwards1973,Edwards1975} which is optimal for infinitely many graphs (see Alon \cite{Alon1996} for more advances).

In past decades, a new area of {\it judicious} partitioning problems has arisen,
where the common theme seeks for partitions optimizing several quantities simultaneously.
A celebrated result in this area is due to Bollob\'as and Scott \cite{Bollobas1999},
who showed that any graph with $m$ edges admits a bipartition $V_1\cup V_2$ of its vertex set such that the number of edges between $V_1$ and $V_2$ is at least $m/2+(\sqrt{2m+1/4}-1/2)/4$ (i.e., Edwards' bound)
and moreover for each $i\in \{1,2\}$, the number of edges contained in $V_i$ is at most $m/4+(\sqrt{2m+1/4}-1/2)/8$.
There has been a large body of research for judicious partitioning problems (see, e.g., \cite{Alon2003, Bollobas2000, Bollobas2004, Fan2017, Has2012, Has2014, Hou2016, ji2018, Lee2013, Lee2014, MaJCTB, Ma2012, Ma2016, por1992, spink, Xu2011,  Xu2014}).
For a systematic treatment on related problems and results, we refer interested readers to the surveys \cite{Bollobas2002_1,Scott2005,WH2022}.

In this paper, we study partitioning problems for directed graphs (i.e., digraphs).
For a digraph $D$ and a bipartition $V(D)=V_1\cup V_2$,
let $e(V_1,V_2)$ denote the size of the {\it directed cut} $(V_1,V_2)$, i.e., the number of arcs in $D$ directed from $V_1$ to $V_2$.
It is easy to find a directed cut of size at least $m/4$ in any digraph with $m$ edges;
on the other hand, Alon, Bollob\'{a}s, Gy\'{a}rf\'{a}s, Lehel and Scott \cite{Alon2007} constructed digraphs whose maximum directed cut is $m/4 + O(m^{4/5})$.
A natural problem in the judicious setting would be to find a bipartition $V(D)=V_1\cup V_2$ that maximizes $\min\{e(V_1, V_2), e(V_2, V_1)\}$.
However, if the digraph $D$ is an oriented star where all arcs are oriented from its center in one direction,
then one of $e(V_1,V_2)$ and $e(V_2,V_1)$ will always be zero for any bipartition $V_1\cup V_2$.
To avoid this situation, Scott \cite{Scott2005} suggested the following problem by imposing a minimum degree condition.

\begin{prob}[Scott \cite{Scott2005}, see Problem 3.27]\label{scottprob}
Let $d$ be a positive integer. What is the largest constant $c_{d}$ such that every digraph $D$ with $m$ arcs and minimum outdegree $d$ admits a bipartition $V(D)=V_1\cup V_2$
such that
\[
\min\left\{e(V_1, V_2), e(V_2, V_1)\right\}\geq c_{d}\cdot m?
\]
\end{prob}

It is easy to obtain $c_1=0$.
Lee, Loh and Sudakov \cite{Lee2014} proved that $c_2=1/6+o(1)$ and $c_3=1/5+o(1)$.
For the general case, they made the following attractive conjecture.

\begin{conj}[Lee, Loh and Sudakov \cite{Lee2014}, see Conjecture 1.2]\label{conj1}
Let $d\geq 3$ be a positive integer. Every digraph $D$ with $m$ arcs and minimum outdegree at least $d+1$ admits a bipartition
$V(D)=V_1\cup V_2$ such that
\begin{align}\label{equ:bounds}
\min\left\{e(V_1,V_2),e(V_2, V_1)\right\}\geq \left(\frac{d}{2(2d+1)}+o(1)\right)m.
\end{align}
\end{conj}

If it is true, this would provide the optimal bound and thus show $c_{d+1}=\frac{d}{2(2d+1)}+o(1)$ for Problem~\ref{scottprob} (see the construction below).
Recently, Liu and Yu \cite{Liu2021} proved the next case of Conjecture~\ref{conj1}, i.e., for digraphs with minimum outdegree at least 4.
It is not clear if the minimum outdegree alone is sufficient for Conjecture~\ref{conj1} for large $d$ and
there have been researches for partitioning digraphs under additional conditions.
A natural additional assumption here is to impose some minimum indegree.
In this line, Hou, Ma, Yu and Zhang \cite{hou} proved Conjecture \ref{conj1} under the additional assumption that the minimum indegree is at least $d+1$.
In other words, the authors of \cite{hou} proved that \eqref{equ:bounds} holds for digraphs with minimum semidegree at least $d+1$.
In \cite{Hou2017}, Hou and Wu proposed the following question for a weaker minimum semidgree condition (also see Conjecture 4.6 in \cite{WH2022}).

\begin{que}[Hou and Wu \cite{Hou2017}]\label{prob}
Is it true that for every positive integer $d$, every digraph $D$ with $m$ arcs
and minimum semidegree $d$ admits a bipartition $V(D)=V_1\cup V_2$ such that
\[
\min\left\{e(V_1,V_2),e(V_2, V_1)\right\}\geq \left(\frac{d}{2(2d+1)}+o(1)\right)m?
\]
\end{que}

In the same paper \cite{Hou2017}, Hou and Wu proved that the above conclusion holds for all oriented graphs with minimum semidegree $d\geq 21$.
Their proofs are involved and work only for oriented graphs.
Later, Hou, Li and Wu \cite{Hou2020} gave a positive answer to Question~\ref{prob} for $d\leq 3$.

A \emph{bisection} of a (di)graph is a bipartition of its vertex set in which the number of vertices in the two parts differ by at most 1.
Finding bisections is of particular interest in partitioning problems (see, e.g., \cite{Fan2018, Hou2021, HWY2017, Lee2013, lin2021, Xu2014}) and it is usually more challenging.

Here we give an affirmative answer to Question~\ref{prob} in a stronger form, where we show that the desired bipartition can be further chosen as a bisection.

\begin{thm}\label{thm1}
Let $d$ be a positive integer. Every digraph $D$ with $m$ arcs and minimum semidegree at least $d$ admits a bisection
$V(D)=V_1\cup V_2$ such that
\[
\min\left\{e(V_1,V_2),e(V_2, V_1)\right\}\geq \left(\frac{d}{2(2d+1)}+o(1)\right)m.
\]
\end{thm}

The proof of Theorem \ref{thm1} at the basic level follows by an approach
launched by Bollob\'as and Scott \cite{BS1997,Bollobas2002_1} and further developed in the subsequent works such as \cite{Hou2017,Lee2013,Lee2014,Ma2012}.
It is to first partition the set $X$ of vertices with large degree by some optimization and then partition the set $Y$ of the remaining vertices using some randomize process
which is carefully designed based on the structures of the cut $(X,Y)$ as well as the subgraph induced by $Y$.
Then using probabilistic concentration inequalities, one can bound the deviations of random performances in terms of parameters associated with the aforementioned structures
(see Lemma~\ref{lem7} for implementing this step in our proof).
For partitioning digraphs, the difficulties are often to untangle the intricate relationships between these parameters.
The new ingredients at the core of our proof are some relatively neat estimations on the arcs between $X$ and $Y$ in each direction (we refer to Subsection~3.3 for more details).

Finally we remark that the bounds in Conjecture \ref{conj1} and Question~\ref{prob} are asymptotically optimal by considering the following digraphs given by Lee, Loh and Sudakov \cite{Lee2014}:
Take $k$ disjoint copies of $K_{2d+1}$ oriented along an Eulerian circuit,
and a single copy of $K_{2d+3}$ oriented in a similar manner which is disjoint from the copies of $K_{2d+1}$;
then fix a vertex $v_0$ of $K_{2d+3}$, and add arcs from all the vertices in the copies of $K_{2d+1}$ to $v_0$.
Clearly, the resulting digraph $D$ has minimum outdegree $d+1$ and minimum semidegree $d$, whose number of arcs equals
\begin{align*}
m= \; kd(2d+1)+(d+1)(2d+3)+k(2d+1)= \; k(d+1)(2d+1)+(d+1)(2d+3).
\end{align*}
For any bipartition of $K_{2d+1}$ oriented along an Eulerian circuit,
the number of arcs in any direction is at most $\frac{d(d+1)}{2}$.
Then for every bipartition $V(D)= V_1 \cup V_2$ (say $v_0\in V_1$), we have
\[
e(V_1,V_2)\leq  k \frac{d(d+1)}{2}+ \frac{(d+1)(d+2)}{2}=\frac{d}{2(2d+1)}m + \frac{(d+1)^2}{2d+1}.
\]
Hence we can conclude that
\[
\max_{V(D)= V_1 \cup V_2}\min\left\{e(V_1,V_2),e(V_2, V_1)\right\}\leq \left(\frac{d}{2(2d+1)}+o(1)\right)m.
\]

The rest of this paper is organized as follows.
In Section 2, we give notations and collect lemmas for the coming proof.
In Section 3, we complete the proof of Theorem \ref{thm1}.
Finally in Section 4, we offer some concluding remarks.

\medskip

\section{Preliminaries}

\subsection{Notations}
We use standard notations in graph theory.
All digraphs considered in this paper are finite with no loops and no multiple arcs (but a pair of arcs in opposite directions is allowed).
For a digraph $D$, we denote by $V(D)$ the {\it vertex set} of $D$ and by $A(D)$ the {\it arc set} of $D$.
Let $x, y$ be two vertices in $D$.
We write $xy$ or $x\to y$ for an arc in $D$ with head $x$ and tail $y$.
We write $N_{D}^{+}(x)=\{y : xy\in A(D)\}$ and $N_{D}^{-}(x)=\{y : yx\in A(D)\}$,
and call $d_{D}^{+}(x):=|N_{D}^{+}(x)|$ and $d_{D}^{-}(x):=|N_{D}^{-}(x)|$ as the {\it outdegree} and {\it indegree} of $x$ in $D$,
respectively.
The {\it degree} of $x$ in $D$ is defined as $d_D(x)= d^{+}_D(x)+d^-_D(x)$,
and the {\it semidegree} of $x$ in $D$ is $\min\{d^{+}_D(x),d^-_D(x)\}$.
We will omit the subscripts in the above definitions when this is not ambiguous.

The {\it maximum degree} $\Delta(D)$ of $D$ is $\max\{d_D(x) : x\in V(D)\}$.
The {\it minimum outdegree} of $D$ is $\delta^{+}(D)= \min\{d_{D}^{+}(x) : x\in V(D)\}$
and the {\it minimum indegree} of $D$ is $\delta^{-}(D)= \min\{d_{D}^{-}(x) : x\in V(D)\}$.
Finally, we define the {\it minimum semidegree} of $D$ to be $\delta^0(D)=\min\{\delta^+(D),\delta^-(D)\}$.

For any $X\subseteq  V(D)$, the subgraph of $D$ induced by $X$ is denoted by $D[X]$.
We write $e(X)$ to express the number of arcs contained in $D[X]$.
For disjoint sets $X, Y\subseteq V(D)$, by $(X,Y)$ we denote the set of all arcs in $D$ with head in $X$ and tail in $Y$.
Let $e(X,Y)$ be the number of arcs in $(X,Y)$.
Throughout the rest of the paper, we write $[k]$ for the set of integers $\{1, 2, \ldots, k\}$ where $k$ denotes a positive integer.

\subsection{Previous lemmas}
In this subsection, we collect some previous lemmas for the use of the coming proof.
First we state the following version of the Azuma-Hoeffding inequality.

\begin{lem}[Azuma \cite{Azuma1967}, Hoeffding \cite{Hoeffding1963}]\label{lem6}
Let $Z_1,\ldots,Z_n$ be independent random variables taking values
in $\{1,\ldots,k\}$, let $Z :=(Z_1,\ldots,Z_n)$, and let $f : \{1,\ldots,k\}^n \rightarrow \mathbb{N}$ such that
$|f(Y)-f(Y')|\le c_i$ for any
$Y,Y'\in \{1,\ldots,k\}^n$ which differ only in the i-th coordinate. Then for any $z>0$,
$$\mathbb{P}(f(Z)\ge\mathbb{E}(f(Z))+z)\; \le \; exp\left(\frac{-z^2}{2\sum_{i=1}^{n}c_i^2}\right),$$
$$\mathbb{P}(f(Z)\le\mathbb{E}(f(Z))-z)\; \le \; exp\left(\frac{-z^2}{2\sum_{i=1}^{n}c_i^2}\right).$$
\end{lem}

A natural, basic idea for finding good cuts is to investigate cuts generated by random partitions.
If we place each vertex of a digraph $D$ to a side uniformly and independently,
then we get a random partition with the expected number of arcs in each direction being $e(D)/4$.
In the following lemma, Lee, Loh and Sudakov \cite{Lee2014} showed that
such a partition (with asymptotically optimal performances) can be realized under some additional restrictions.

\begin{lem}[Lee, Loh and Sudakov \cite{Lee2014}]\label{lem1}
Let $D$ be a digraph with $n$ vertices and $m$ arcs. For every $\varepsilon>0$,
if $m\ge 8n/\varepsilon^2$ or $\Delta(D) \le \varepsilon^2m/4$, then $D$ admits a bipartition $V(D)=V_1\cup V_2$ such that
$$\min\{e(V_1, V_2), e(V_2, V_1)\}\ge\frac{1}{4}m-\varepsilon m.$$
\end{lem}

This lemma was extended by Hou, Wu and Yan \cite{HWY2017} in the following bisection version.

\begin{lem}[Hou, Wu and Yan \cite{HWY2017}]\label{lem4}
Let $D$ be a digraph with $n$ vertices and $m$ arcs. For every $\varepsilon>0$, if $m\ge 16n/\varepsilon^2$ or $\Delta(D)\le$
$\varepsilon^2m/8$, then $D$ admits a bisection $V(D)=V_1\cup V_2$ such that
$$\min\{e(V_1,V_2),e(V_2,V_1)\}\ge\frac{1}{4}m-\varepsilon m.$$
\end{lem}

We now define an important structural concept in this study introduced by Lee, Loh and Sudakov \cite{Lee2013}.
Let $T$ be a connected component of a given (undirected) graph $G$. We say that $T$ is a {\it tight} component of $G$ if it satisfies the following properties:
\begin{itemize}\setlength{\itemsep}{-1.5pt}
\item [(i)] for every vertex $v\in V (T)$, $T-v$ contains a perfect matching, and
\item [(ii)]  for every vertex $v\in V (T)$ and every perfect matching $M$ of $T-v$, no edge in $M$
has exactly one end adjacent to $v$.
\end{itemize}
Clearly a tight component must have an odd number of vertices.
The {\it underlying graph} $G_D$ of a digraph $D$ is a simple graph obtained from $D$ by ignoring arc orientations and
removing redundant multi-edges.
We say a component of $D$ is {\bf tight} if the corresponding component of $G_D$ is tight,
and {\bf essential} if it is tight and does not contain two arcs in opposite directions.

Let $\{e_1,\ldots, e_s\}$ be a maximum matching in a graph $G$, and let $W$ be the set of vertices not in the matching.
Let $v\in V(e_i)$ and $w\in W$.
With respect to this fixed matching, $v$ is called a {\it free} neighbor of $w$
if $w$ is adjacent to $v$ but not adjacent to the other endpoint of $e_i$.
We call $w\in W$ a {\it free} vertex if it has at least one free neighbor in the matching, and a {\it non-free} vertex otherwise.
In \cite{Lee2014} Lee, Loh and Sudakov showed that there is a bijective correspondence between non-free vertices and tight components under certain conditions.

\begin{lem}[Lee, Loh and Sudakov \cite{Lee2014}]\label{lem8}
Let $e_1,\ldots, e_s$ be the edges of a maximum matching in an undirected graph $G$, and let $W$ be the set of vertices not in the matching.
Further assume that among all matchings of maximum size, we have chosen one which maximizes the number of free vertices in $W$.
Then, every tight component contains a distinct non-free vertex of $W$, and all non-free vertices of $W$ are covered in this way (i.e., there is a bijective correspondence).
\end{lem}

Using this lemma, the authors of \cite{Lee2014} (see its Lemma 3.2) managed to decompose an undirected graph into induced stars plus an independent set.
We need the following version of this decomposition result which can be easily deduced from the original proof of \cite{Lee2014}.\footnote{In the original statement of Lemma 3.2 in \cite{Lee2014}, $G$ is a graph with $n$ vertices and $m\leq Cn$ edges and $A$ denotes the set of vertices with degree at least $2C/\varepsilon$, where $\varepsilon$ and $C$ are arbitrary positive reals. In particular $|A|\leq \varepsilon n$.}

\begin{lem}[Lee, Loh and Sudakov \cite{Lee2014}, Lemma 3.2]\label{lem5}
Let $G$ be an undirected graph with $\tau$ tight components.
Let $A$ be a subset of $V(G)$.
Let $e_1,\ldots, e_s$ be the edges of a maximum matching in $G$ such that among all matchings of maximum size,
we have chosen one which maximizes the number of free vertices.
Then there exists
a partition $V(G)=T_1\cup T_2\cup \cdots\cup T_s\cup U$ such that
\begin{itemize}\setlength{\itemsep}{-1.5pt}
       \item [$(i)$]  each $T_i$ induces a star containing the edge $e_i$ in the maximum matching,
       \item [$(ii)$] at most one leaf vertex in each $T_i$ is contained in $A$, and
       \item [$(iii)$] $U$ is an independent set of order $|U|\le \tau+|A|$.
   \end{itemize}
\end{lem}

The following lemma of \cite{Lee2014} and its variances serve as a guiding tool in the recent development of digraph partitioning problems (see, i.e., \cite{Lee2014,HWY2017,Hou2020}).

\begin{lem}[Lee, Loh and Sudakov \cite{Lee2014}]\label{lem3}
For any positive constants $C$ and $\varepsilon$, there exist $\gamma$, $n_{0}>0$ for which the following holds.
Let $D$ be a digraph with $n\geq n_{0}$ vertices and at most $Cn$ arcs. Suppose $X\subseteq V(D)$ is a set
of at most $\gamma n$ vertices and $X_1,X_2$ is a partition of $X$.
Let $Y=V(D)\setminus X$ and let $\tau$ be the number of components with odd order in $D[Y]$.
If every vertex in $Y$ has degree at most
$\gamma n$ in $D$, then there is a bipartition
$V(D)=V_{1}\cup V_{2}$ with $X_{i}\subseteq V_{i}$ for $i=1,2$ such that
$$e(V_{1},V_{2})\geq e(X_{1},X_{2})+\frac{e(X_{1},Y)+e(Y,X_{2})}{2}+\frac{e(Y)}{4}+\frac{n-\tau}{8}-\varepsilon n,$$
$$e(V_{2},V_{1})\geq e(X_{2},X_{1})+\frac{e(X_{2},Y)+e(Y,X_{1})}{2}+\frac{e(Y)}{4}+\frac{n-\tau}{8}-\varepsilon n.$$
\end{lem}

We mention that Hou, Wu and Yan \cite{HWY2017} extended this lemma by replacing $\tau$ with the number of tight components in $D[Y]$ and requiring the bipartition $V(D)=V_1\cup V_2$ to be a bisection,
and more recently, Hou, Li and Wu \cite{Hou2020} improved it by showing $\tau$ can be chosen as the number of essential components in $D[Y]$.
To show the desired bisection of Theorem \ref{thm1}, we will need a further strengthening of Lemma \ref{lem3} for which we prove in the following subsection.

\subsection{A strengthening of Lemma \ref{lem3}}
In this subsection, using probabilistic arguments we prove Lemma~\ref{lem7}, which generalizes Lemma \ref{lem3} as well as its variances mentioned above.
This lemma will allow us to extend a pre-optimized partition of the set of vertices with large degree to a `good' bisection of the entire digraph.
We like to point out that it is crucial for the proof in Section~3 to use $\tau$ as the number of essential components.

Since the proof is similar to that of Lemma 3.1 in \cite{Lee2014} and Lemma 3.5 in \cite{HWY2017},
we shall only give detailed explanations for steps which reflect the difference and sketch the rest of them.

\begin{lem}\label{lem7}
For any positive constants $C$ and $\varepsilon$, there exist $\gamma$, $n_{0}>0$ for which the following holds.
Let $D$ be a digraph with $n\geq n_{0}$ vertices and at most $Cn$ arcs. Suppose $X\subseteq V(D)$ is a set
of at most $\gamma n$ vertices and $X_1,X_2$ is a partition of $X$.
Let $Y=V(D)\setminus X$ and let $\tau$ be the number of {\bf essential} components in $D[Y]$.
If every vertex in $Y$ has degree at most
$\gamma n$ in $D$, then there is a bisection
$V(D)=V_{1}\cup V_{2}$ with $X_{i}\subseteq V_{i}$ for $i=1,2$ such that
$$e(V_{1},V_{2})\geq e(X_{1},X_{2})+\frac{e(X_{1},Y)+e(Y,X_{2})}{2}+\frac{e(Y)}{4}+\frac{n-\tau}{8}-\varepsilon n,$$
$$e(V_{2},V_{1})\geq e(X_{2},X_{1})+\frac{e(X_{2},Y)+e(Y,X_{1})}{2}+\frac{e(Y)}{4}+\frac{n-\tau}{8}-\varepsilon n.$$
\end{lem}

\begin{proof}
Let $G$ be the underlying graph of $D$.
Let $\tau^*$ be the number of tight components in $G[Y]$, and let $\sigma$ be the number of tight components in $G[Y]$ which contains an antiparallel pair when lifted to $D$.
Clearly, $\tau=\tau^*-\sigma$ is the number of essential components in $D[Y]$.

We choose $e_1,\ldots, e_s$ to be the edges of a maximum matching in $G$ such that
\begin{itemize}
\item [(1).] among all matchings of maximum size, we choose one which maximizes the number of free vertices in the set $W$ of vertices not in the matching, and
\item [(2).] subject to the above, we require that the number of edges $e_i$ which corresponds to a pair of arcs in opposite directions when lifted to $D$ (call $e_i$ {\it special}) is maximized.
\end{itemize}
We claim that the number of special edges in $\{e_1,\ldots, e_s\}$ is at least $\sigma$.
It suffices to show that each tight component $B$ in $G[Y]$ which contains an antiparallel pair when lifted to $D$ contributes at least one special edge.
Suppose that $B$ has $2k+1$ vertices and $D[B]$ contains $u\to v$ and $v\to u$.
It is easy to see that any maximum matching in $G$ contains $k$ edges in $B$ and for any matching of $k$ edges in $B$, the remaining vertex of $B$ is a
non-free vertex.
So taking any matching of $k$ edges in $B$ will not affect the property $(1)$.
Now we just need to show that there exists a matching of $k$ edges in $B$ containing the edge $uv$,
which can be obtained from any perfect matching $M$ of $B-v$ by replacing the edge of $M$ containing $u$ with the edge $uv$.

Let $A=\{x\in V(G): d_G(x)\geq 2C/\varepsilon\}$. Since $e(G)\leq e(D)\leq Cn$, we have $|A|\leq \varepsilon n$.
Applying Lemma \ref{lem5} to $G[Y]$ with respect to $A$ and $\{e_1,\ldots, e_s\}$,
we obtain a partition $Y=T_1\cup\cdots\cup T_s\cup U$
such that each $T_i$ induces a star containing the edge $e_i$,
at most one leaf vertex in each $G[T_i]$ has degree at least $2C/\varepsilon$ in the whole graph $G$,
and $U$ is an independent set with $|U|\leq\tau^* +\varepsilon n$.

We then randomly construct a partition $V(D)= V_1 \cup V_2$ as follows:
place each $X_i$ in $V_i$ for $i=1,2$; partition each $T_i$ by independently placing its center vertex $v_i$ on a uniformly random side, and then placing the rest of $T_i\backslash \{v_i\}$ on the other side;
place each remaining vertex (from the set $U$) on a uniformly random side.

Define random variables $Y_1=e(V_1,V_2)$ and $Y_2=e(V_2,V_1)$.
For an arc $e=u\to v$, let $I_e$ be the indicator random variable of the event that $u\in V_1$ and $v\in V_2$.
Then $Y_1=\sum_{e\in A(D)}I_e$.
We see that $\mathbb{E}[I_e]=1$ if $u\in X_1$ and $v\in X_2$ and $\mathbb{E}[I_e]=1/2$ if either
$u\in X_1$ and $v\in Y$, or $u\in Y$ and $v\in X_2$.
For an arc $e$ in $D[Y]$, if $e$ is an arc induced by some set $T_i$, then $\mathbb{E}[I_e]=1/2$;
otherwise, $\mathbb{E}[I_e]=1/4$.
Note that any special edge $e_i$ contributes two arcs induced by $T_i$,
so the total number of arcs induced by the sets $T_i$ is at least $\frac{|Y|-|U|}{2}+\sigma\geq \frac{(n-\gamma n)-(\tau^*+\varepsilon n)}{2}+\sigma$.
Therefore, as $\tau=\tau^*-\sigma$, we have
\begin{align*}
\mathbb{E}[Y_1]\ge \; &  e(X_1,X_2)+\frac{e(X_1,Y)+e(Y,X_2)}{2}+\frac{e(Y)}{4}+\frac{1}{4}
                      \left[\frac{(n-\gamma n)-(\tau^*+\varepsilon n)}{2}+\sigma\right]\\[2.5mm]
\ge \; & e(X_1,X_2)+\frac{e(X_1,Y)+e(Y,X_2)}{2}+\frac{e(Y)}{4}+\frac{n-\tau}{8}-\frac{(\varepsilon+\gamma)n}{8},
\end{align*}
where we choose $\gamma\ll \varepsilon, C$. Similarly, we can derive that
\[
\mathbb{E}[Y_2]\ge e(X_2,X_1)+\frac{e(X_2,Y)+e(Y,X_1)}{2}+\frac{e(Y)}{4}+\frac{n-\tau}{8}-\frac{(\varepsilon+\gamma)n}{8}.
\]

Now we have reached the same position as in Lemma 3.1 of \cite{Lee2014} and Lemma 3.5 of \cite{HWY2017}.
Following the same arguments therein, one can use Lemma~\ref{lem6} (the Azuma-Hoeffding inequality) to control the deviations of random variables $Y_1, Y_2, |V_1|, |V_2|$ of the randomized partition $V_1\cup V_2$.
This leads to an almost bisection $V_1\cup V_2$ with desired bounds on $e(V_1, V_2)$ and $e(V_2,V_1)$, where $|V_i-n/2|=o_{\varepsilon, C}(n)+ \gamma n$.
Since $e(D)\leq Cn$, the number of vertices with degree at most $4C$ is at least $n/2$.
As shown in \cite{HWY2017}, one can equalize $|V_1|$ and $|V_2|$ by moving around $o_{\varepsilon, C}(n)+ \gamma n$ vertices with degree at most $4C$,
which will only affect $e(V_1, V_2)$ and $e(V_2,V_1)$ by at most $\varepsilon n/2$ (as $\gamma\ll \varepsilon, C$).
This completes the proof of Lemma~\ref{lem7}.
\end{proof}

\section{Proof of Theorem~\ref{thm1}}
Let $d$ be a positive integer and let $D$ be a digraph with $n$ vertices and $m$ arcs whose minimum semidegree is at least $d$.
Let $\varepsilon$ be any positive real number and $m$ be sufficiently large compared with $d, \varepsilon$.
Our goal is to find a bisection $V(D)=V_1\cup V_2$ such that
\begin{equation}\label{equ:main}
\min\{e(V_1,V_2),e(V_2, V_1)\}\geq \left(\frac{d}{2(2d+1)}-\varepsilon\right)m.
\end{equation}
As $n(n-1)\ge m$, we may also assume that $n$ is sufficiently large compared with $d, \varepsilon$
(so that Lemma~\ref{lem7} can be applied later).
If $m\geq 256(2d+1)^2 n$, then applying Lemma \ref{lem4} with $\varepsilon=\frac{1}{4(2d+1)}$ yields a bisection $V(D)=V_1\cup V_2$ such that
\[
\min\{e(V_1,V_2),e(V_2,V_1)\}\; \geq \;\frac{m}{4}-\frac{m}{4(2d+1)}=\frac{d}{2(2d+1)}m,
\]
as desired.
So in the rest of the proof we may assume that (also as the minimum outdegree is at least $d$)
\begin{equation}\label{equ:0}
	dn\leq m< 256(2d+1)^2 n.
\end{equation}

\subsection{The initial partition $V(D)=X\cup Y$}
Throughout the rest of the proof,
we consider the partition $V(D)=X\cup Y$ such that $X:=\{v\in
V(D):d_D(v)\geq n^{3/4}\}$ and $Y:=V(D)\setminus X$. Then, by \eqref{equ:0},
\[
|X|\cdot n^{3/4}\leq \sum_{v\in X}d_D(v)\;\leq \sum_{v\in V(D)}d_D(v)=2m< 512(2d+1)^2 n.
\]
Hence, $|X|=O(m/n^{3/4})=O(n^{1/4})$, and thus  $$e(X)\leq|X|^{2}=O(n^{1/2})=O(m^{1/2}).$$

We emphasize that in the following proof, we will only use the semidegree condition for vertices in $Y$.
For the sake of simplicity we remove all arcs within $X$, and update $m$ to be the new total number of arcs in the resulting digraph.\footnote{Note that this will not change the indegree and outdegree of each vertex in $Y$.}
Therefore from now on, we may assume that
\begin{align*}\label{equ:1}
	e(X)=0.
\end{align*}
Note that this will only affect $O(m^{1/2})$ amount of arcs.
So to prove \eqref{equ:main}, it suffices to show that the resulting digraph (which we still call $D$) has a bisection $V(D)=V_1\cup V_2$
for which
\begin{equation}\label{equ:main2}
\min\{e(V_1,V_2),e(V_2,V_1)\}\ge\Big(\frac{d}{2(2d+1)}-\frac{\varepsilon}{2}\Big)m.
\end{equation}

We will first find a good bipartition for the set of vertices of large degree $X$
and we begin by introducing some useful concepts.
As we shall see later, the vertices in $X$ whose outdegree and indegree differ significantly play important roles.
For $v\in X$, we define
$$s^+(v):=d^+(v)-d^-(v), ~~ s^-(v):=d^-(v)-d^+(v), ~ \mbox{ and }
 ~ s(v):=|d^+(v)-d^-(v)|.$$
For a partition $X=X_1\cup X_2$,
its ordered {\it gap} is defined as
$$\theta(X_1,X_2)=  \Big(e(X_{1},Y)+e(Y,X_{2})\Big)-\Big(e(X_{2},Y)
+e(Y,X_{1})\Big).$$
Since $e(X)=0$, we also see that
\begin{equation}\label{theta}
\begin{aligned}
\theta(X_1,X_2) = \sum_{v\in X_1}s^+(x)-\sum_{v\in X_2}s^+(x).
\end{aligned}
\end{equation}
For a given partition $X=X_1\cup X_2$, we say a vertex $v\in X$ is {\it forward} if either $v\in X_1$ and $s^+(v)>0$, or $v\in X_2$ and $s^-(v)>0$;
otherwise, we say $v$ is {\it backward}.
Let $X_f:=\{v\in X: v \mbox{ is forward}\}$ and $X_b:=\{v\in X: v \mbox{ is backward}\}$.
Then we can rewrite \eqref{theta} as the following
\begin{equation}\label{theta1}
\theta(X_1,X_2)=\sum_{v\in X_f}s(x) - \sum_{v\in X_b}s(x).
\end{equation}

\subsection{Partitioning $X$}
In the rest of this section, we choose the partition $X=X_1\cup X_2$ such that $|\theta(X_1,X_2)|$ is minimum.
Note that $\theta(X_1,X_2)=-\theta(X_2,X_1)$.
So by swapping $X_1$ and $X_2$ if necessary, we may assume $\theta(X_1,X_2)\geq 0$.
In the following, we write $\theta=\theta(X_1,X_2)$, unless otherwise specified.

Let $\tau$ be the number of essential components in $D[Y]$.
Note that $n$ is sufficiently large compared with $d, \varepsilon$.
So applying Lemma \ref{lem7} with $C :=256(2d+1)^2$, $D$ has a bisection
$V(D)=W_1\cup W_2$ such that $X_i\subseteq W_i$ for $i=1,2$, and moreover (note that $e(X)=0$)
\begin{align*}
&\min \{e(W_1,W_2),e(W_2,W_1)\}\\[0.5mm]
\ge\; & \frac{1}{2}\min\{e(X_1,Y)+e(Y,X_2),e(X_2,Y)+e(Y,X_1)\}+\frac{e(Y)}{4}+\frac{n-\tau}{8}-\frac{\varepsilon}{4} n\\[1.7mm]
=\; & \frac{e(X,Y)+e(Y,X)-\theta}{4}+\frac{e(Y)}{4}+\frac{n-\tau}{8}-\frac{\varepsilon}{4} n\\[0.7mm]
=\; & \frac{m-\theta}{4}+\frac{n-\tau}{8}-\frac{\varepsilon}{4} n\\[1.5mm]  
\ge\; & \frac{m-\theta}{4}+\frac{n-\tau}{8}-\frac{\varepsilon}{2} m,  
\end{align*}
where the last inequality holds as $m\geq dn/2\geq n/2$ from \eqref{equ:0}.
To prove \eqref{equ:main2} (and thus Theorem~\ref{thm1}), it suffices to show that
$$\frac{m-\theta}{4}+\frac{n-\tau}{8}-\frac{\varepsilon}{2}m\ge \Big(\frac{d}{2(2d+1)}-\frac{\varepsilon}{2}\Big)m,$$
which is equivalent to the following
\begin{equation}\label{equ:main3}
\frac{m}{2d+1}+\frac{n}{2}-\theta-\frac{\tau}{2}\ge 0.
\end{equation}
Note that we have $\tau\le n$. If $\theta\leq \frac{m}{2d+1}$, then \eqref{equ:main3} holds trivially.
So we may assume
\begin{equation}\label{eq0}
\theta>\frac{m}{2d+1}>0.
\end{equation}
We point out that since $\theta>0$, it follows from \eqref{theta1} that $X$ has at least one forward vertex.

Our strategy in the coming proof then is to prove \eqref{equ:main3} by establishing a sequence of claims.
The following two claims can be found in \cite{Hou2020} whose proofs can be traced back to \cite{Lee2014}.
For the sake of completeness, we include their short proofs here.

\begin{claim}\label{cla1}
Call a vertex $v\in X$ {\bf huge} if $s(v)\ge \theta$. Then all forward vertices in $X$ are huge.
In particular, there is at least one huge vertex.
\end{claim}

\begin{proof}
As we just point out after \eqref{eq0}, there is at least one forward vertex in $X$.
Take any forward vertex $v\in X$.
By definition, $s(v)>0$.
Let $X'_1\cup X'_2$ be the new partition of $X$ obtained from $X_1\cup X_2$ by switching the side of $v$. Then
\begin{align*}
\theta(X'_1,X'_2)=\; &\big(e(X'_1,Y)+e(Y,X'_2)\big)-\big(e(X'_2,Y)+e(Y,X'_1)\big)\\[2mm]
 =\; & \big(e(X_1,Y)+e(Y,X_2)-s(v)\big)-\big(e(X_2,Y)+e(Y,X_1)+s(v)\big)=\theta-2s(v).
\end{align*}
Since $X_1\cup X_2$ is a partition of $X$ minimizing $|\theta(X_1, X_2)|$,
we have $|\theta(X'_1,X'_2)|=|\theta-2s(v)|\geq \theta$.
This forces $s(v)\geq \theta$, which means $v$ is huge.
\end{proof}

Let $X'$ be the set of huge vertices in $X$.
By Claim \ref{cla1}, we know that all vertices in $X\backslash X'$ are backward vertices.
Define $$\alpha=|X'| ~~  \mbox{ and } ~~ g=\sum_{v\in X\backslash X'}s(v).$$

\begin{claim}\label{cla2}
	For any forward vertex $v\in X$, we have $s(v)\geq \theta + g$.
\end{claim}

\begin{proof}
For the bipartition $X=X_1\cup X_2$,
we call an arc in $(X_1,Y)\cup (Y,X_2)$ {\it forward} and an arc in $(X_2,Y)\cup (Y,X_1)$ {\it backward}.
Let $m_f=e(X_1,Y)+e(Y,X_2)$ and $m_b=e(X_2,Y)+e(Y,X_1)$. Note that $\theta=m_f-m_b$.

Suppose to the contrary that there exists a forward vertex $v\in X$ such that $g>s(v)-\theta\ge 0$.
Recall that $g=\sum_{v\in X\backslash X'} s(v)$.
Then $X\backslash X'=\{u\in X : s(u)<\theta\}\neq \emptyset$.
Now we switch vertices between $X_1$ and $X_2$ in two consecutive rounds.
First, we move $v$ to the other part of the partition $X=X_1\cup X_2$.
Then the number of forward arcs in the resulting partition of $X$
becomes $m_f'=m_f-s(v)\le m_f-\theta=m_b$.
Next, we choose a vertex in $X\backslash X'$ (which remains backward) at a time and switch it to the other part.
After we switch all vertices in $X\backslash X'$,
the number of forward arcs becomes $m_f''=m_f'+g=m_f-s(v)+g> m_f-s(v)+(s(v)-\theta)=m_f-\theta=m_b$.

In the second round, when the first time the number of forward arcs $m_f^*$ is greater than $m_b$, it must be the case that $m_f^*< m_b+\theta=m_f$.
Note that during the above switching process, the number of forward and backward arcs in total always equals to $m_f+m_b$.
This tells that at the very same moment, the number of backward arcs $m_b^*$ also satisfies $m_b< m_b^*<m_f$.
So we get $|m_f^*-m_b^*|<|m_f-m_b|=\theta$, contradicting the minimality of the gap $\theta$.
This proves that $s(v)\geq \theta + g$.
\end{proof}

\subsection{Bounding the arcs between $X$ and $Y$}
We turn to some key estimations on the sizes of $(X,Y)$ and $(Y,X)$,
which provide new ingredients to Question~\ref{prob}.
Recall that $\tau$ denotes the number of essential components in $D[Y]$.

\begin{claim}\label{cla3}
$(2d+1)\tau\le |Y|+2\cdot\min\{e(X,Y),e(Y,X)\}$.
\end{claim}

\begin{proof}
For $i=1, 3, \ldots, 2d-1$, let $\tau_i$ be the number of essential components in $D[Y]$ of order $i$;  and let $\tau'$ be the
number of essential components in $D[Y]$ of order at least $2d+1$. Then
\begin{equation}\label{eq1}
\tau=\sum_{i=1}^{d}\tau_{2i-1}+\tau'.
\end{equation}
Considering the number of vertices in $Y$, it yields that
\begin{equation}\label{eq2}
\sum_{i=1}^{d}(2i-1)\tau_{2i-1}+(2d+1)\tau'\le |Y|.
\end{equation}
Now we count the number of arcs from $Y$ to $X$. For each essential
component $D_i$ of order $i$,
as it does not contain two arcs in opposite directions,
we have $e(D_i)\le i(i-1)/2$.
Thus, since the outdegree of every vertex in $Y$ is at least $d$,
we see that $e(D_i,X)\ge d i-i(i-1)/2$. Viewing $f(i)=d i-i(i-1)/2$ as a function of $i$
over the interval $[1, 2d-1]$, we see
that it achieves its minimum at $i=1$ (since $d\ge 1$). Hence, $e(D_i,X)\ge f(1)=d$ for every $i\in [2d-1]$.
Thus,
\begin{equation}\label{eq3}
\sum_{i=1}^{d} d\cdot \tau_{2i-1} \leq e(Y,X).
\end{equation}
Combining \eqref{eq1}, \eqref{eq2} with \eqref{eq3}, we can derive that
$$(2d+1)\tau\;\le\; \sum_{i=1}^{d}(2d+2i-1)\tau_{2i-1}+(2d+1)\tau' \le |Y|+2e(Y,X).$$
Note that the indegree of each vertex in $Y$ is also at least $d$.
Using the same arguments as above, counting the number of arcs from $X$ to $Y$ would lead to
$$(2d+1)\tau\le |Y|+2e(X,Y),$$
proving the claim.
\end{proof}

For convenience, we define
\[
e_{X,Y}:=e(X,Y)+e(Y,X) ~~ \mbox{and} ~~b=\sum_{v\in X} \min\{d^+(v),d^-(v)\}.
\]
For any $v\in X$, we have $d(v)-s(v)=2\min\{d^+(v),d^-(v)\}.$
Therefore
\begin{equation}\label{eq00}
m\geq m-e(Y)=e_{X,Y}=\sum_{v\in X} d(v)=\sum_{v\in X}s(v) +2b=\sum_{v\in X'}s(v) +g +2b\geq \alpha\theta,
\end{equation}
where $\alpha=|X'|$ denotes the number of huge vertices in $X$.
By Claim \ref{cla1}, we have
\begin{equation}\label{tau}
1\le\alpha\leq \frac{m}{\theta}< 2d+1,
\end{equation}
where the second inequality holds because of \eqref{eq00} and the third inequality follows from \eqref{eq0}.

Let $X'=\{v_1,v_2,\ldots,v_{\alpha}\}$
with $s(v_i)=\Delta_i$ for $i\in [\alpha]$ such that $\Delta_{1}\ge\Delta_{2}\ge\cdots\ge\Delta_{\alpha}\ge\theta$.
After \eqref{eq0} we point out that there exists at least one forward vertex (say $v_i\in X'$).
By Claim~\ref{cla2}, $\Delta_i=s(v_i)\geq \theta+g$ and thus we have
\begin{equation}\label{g}
0\leq g\le \Delta_1-\theta.
\end{equation}
Throughout we let $\beta=\lceil\alpha/2\rceil$.
By \eqref{tau}, we see $1\leq \beta\leq d$.

\begin{claim}\label{cla-new}
$\max\{e(X,Y),e(Y,X)\}\geq \beta\theta+b$ and $\min\{e(X,Y),e(Y,X)\}\leq \sum_{i=1}^\alpha \Delta_i +g-\beta\theta+b$.
\end{claim}
\begin{proof}
There always exist $\beta$ vertices in $X'$, say $X^*=\{v_{j_1},v_{j_2},\ldots,v_{j_{\beta}}\}$,
satisfying either (i) $s(v_{j_i})=s^+(v_{j_i})$ for every $i\in [\beta]$, or (ii) $s(v_{j_i})=s^-(v_{j_i})$ for every $i\in [\beta]$.

If the case (i) occurs,
then $s(v_{j_i})=s^+(v_{j_i})=d^+(v_{j_i})-d^-(v_{j_i})\geq \theta$ for $i\in [\beta]$, implying
\begin{equation*}\label{XY}
\begin{aligned}
e(X,Y)=\; &\sum_{v\in X}d^+(v)\\[1mm]
 =\; & \sum_{v\in {X^*}}d^+(v)+\sum_{v\in {X\backslash X^*}}d^+(v)\\[1mm]
\ge\; & \sum_{v\in {X^*}}(d^-(v)+\theta)+\sum_{v\in {X\backslash X^*}}d^+(v)\\[1mm]
\ge\; & \beta\theta+\sum_{v\in X}\min\{d^+(v),d^-(v)\}\\[1.5mm]
 =\; & \beta\theta+b.
\end{aligned}
\end{equation*}
If the case (ii) occurs, by the similar arguments we can obtain that
$e(Y,X)\ge \beta\theta+b.$
This shows that $\max\{e(X,Y),e(Y,X)\}\geq \beta\theta+b$.
Finally, using \eqref{eq00} we have
\[
\max\{e(X,Y),e(Y,X)\}+\min\{e(X,Y),e(Y,X)\}=e_{X,Y}=\sum_{v\in X} s(v)+2b=\sum_{i=1}^\alpha \Delta_i+g+2b,
\]
which implies that
$\min\{e(X,Y),e(Y,X)\}\leq \sum_{i=1}^\alpha \Delta_i +g-\beta\theta+b$, as desired.
\end{proof}

The next claim gives a further upper bound on $\min\{e(X,Y),e(Y,X)\}$.

\begin{claim}\label{cla5}
$\min\{e(X,Y),e(Y,X)\}\le 2\beta\Delta_1-(\beta+1)\theta+b,$ where $1\leq \beta\leq d$.
\end{claim}

\begin{proof}
For clarity of presentation we divide this proof into two cases according to the parity of $\alpha$.
First assume that $\alpha$ is odd. Then $\alpha=2\beta-1$ where $\beta\in [1,d]$.
It follows from Claim~\ref{cla-new} and \eqref{g} that
\begin{align*}
\min\{e(X,Y),e(Y,X)\}\leq \sum_{i=1}^{2\beta-1} \Delta_i +g-\beta\theta+b\leq 2\beta\Delta_1-(\beta+1)\theta+b.
\end{align*}

It remains to consider when $\alpha$ is even. Then $\alpha=2\beta$, where $\beta\in [1,d]$.
Recall \eqref{theta1} and the definitions of $X_f, X_b$.
Let $v_{p_1},v_{p_2},\ldots,v_{p_s}$ be all forward vertices in $X'$ and $v_{q_1},v_{q_2},\ldots,v_{q_t}$ be all backward vertices in $X'$, where $s+t=|X'|=2\beta$.
By Claim \ref{cla1},
all vertices in $X\backslash X'$ are backward vertices. Hence,
\[
X_f=\{v_{p_1},v_{p_2},\ldots,v_{p_s}\} ~ \mbox{ and }~
X_b=\{v_{q_1},v_{q_2},\ldots,v_{q_t}\}\cup X\backslash X'.\
\]
Then by \eqref{theta1} and the definition of $g$, we have
\begin{equation}\label{eq7}
\theta=\sum_{v\in X_f}s(x) - \sum_{v\in X_b}s(x)=\sum_{i=1}^{s}\Delta_{p_i}-\sum_{i=1}^{t}\Delta_{q_i}-g.
\end{equation}
By Claim~\ref{cla-new}, we get
\begin{equation}\label{eq-min}
\min\{e(X,Y),e(Y,X)\}\leq \sum_{i=1}^\alpha \Delta_i +g-\beta\theta+b=\sum_{i=1}^{s}\Delta_{p_i}+\sum_{i=1}^{t}\Delta_{q_i}+g-\beta\theta+b
\end{equation}
If $s\le \beta$, then using \eqref{eq7} the above inequality implies
\[
\min\{e(X,Y),e(Y,X)\}\leq 2\sum_{i=1}^{s}\Delta_{p_i}-\theta-\beta\theta+b\leq 2\beta\Delta_1-(\beta+1)\theta+b,
\]
as desired.
So $s\ge \beta+1$. Then $t=2\beta-s\le \beta-1$.
Using the equivalent form $\sum_{i=1}^{s}\Delta_{p_i}=\theta+\sum_{i=1}^{t}\Delta_{q_i}+g$ of \eqref{eq7},
the inequality \eqref{eq-min} gives that
\[
\min\{e(X,Y),e(Y,X)\}\leq 2\sum_{i=1}^{t}\Delta_{q_i}+2g+\theta-\beta\theta+b\leq 2(t+1)\Delta_1-(\beta+1)\theta+b\leq 2\beta\Delta_1-(\beta+1)\theta+b,
\]
where the second last inequality holds by \eqref{g}.
This claim now is complete.
\end{proof}

\subsection{Completing the proof}
Finally, we are ready to complete the proof by showing \eqref{equ:main3}.

We will need one more inequality that
\begin{equation}\label{eq8}
m\ge \beta\theta+b+d|Y|.
\end{equation}
To see this, since $\delta^0(D)=\min\{\delta^+(D),\delta^-(D)\}\ge d$,
we have $e(Y,X)+e(Y)\ge d|Y|$ and $e(X,Y)+e(Y)\ge d|Y|.$
That says, $$\min\{e(X,Y),e(Y,X)\}+e(Y)\geq d|Y|.$$
Indeed, this together with Claim~\ref{cla-new} give the desired inequality that
$$m=e_{X,Y}+e(Y)=\max\{e(X,Y),e(Y,X)\}+\min\{e(X,Y),e(Y,X)\}+e(Y)\ge \beta\theta+b+d|Y|.$$

We can also derive from Claims~\ref{cla3} and \ref{cla5} that
\begin{equation}\label{eq9}
	(2d+1)\tau\le |Y|+4\beta\Delta_1-2(\beta+1)\theta+2b.
\end{equation}
Since $e(X)=0$, we see $|Y|\geq |d^+(v_1)-d^-(v_1)|=s(v_1)=\Delta_1$.
This together with \eqref{g} imply that $n\ge |Y|\ge \Delta_1\ge \theta$.
Also note that $d\geq \beta$.
Then using \eqref{eq8} and \eqref{eq9} we have
\begin{align*}
    & 2(2d+1)\left(\frac{m}{2d+1}+\frac{n}{2}-\theta-\frac{\tau}{2}\right)\\[1mm]
\ge\; & 2m+(2d+1)|Y|-(4d+2)\theta-(2d+1)\tau \\[1mm]
\ge\; & \Big(2\beta\theta+2b+2d|Y|\Big)+(2d+1)|Y|-(4d+2)\theta-\Big(|Y|+4\beta\Delta_1-2(\beta+1)\theta+2b\Big)\\[1mm]
 =\; & 4d|Y|-4\beta\Delta_1-4(d-\beta)\theta \\[1mm]
\ge\; & 0.
\end{align*}
Hence \eqref{equ:main3} holds. This completes the proof of Theorem \ref{thm1}.
\qedB

\section{Concluding remarks}
In this paper, we answer Question~\ref{prob} by showing that every digraph $D$ with $m$ arcs
and minimum semidegree $d$ admits a bisection $V(D)=V_1\cup V_2$ such that $\min\{e(V_1,V_2),e(V_2, V_1)\}\geq \frac{d}{2(2d+1)}m+o(m)$.
Most arguments in Section~3 in fact work for any digraph which may be useful for Conjecture~\ref{conj1} (to be precise, only Claim~\ref{cla3} and Subsection~3.4 use the minimum semidegree condition).
We also wonder if the following can be true.

\begin{que}
Is it true that for every integer $d\ge 1$ and sufficiently large $m$, every digraph $D$ with $m$ arcs
and minimum semidegree $d$ admits a bipartition $V(D)=V_1\cup V_2$ such that
\[
\min\left\{e(V_1,V_2),e(V_2, V_1)\right\}\geq \frac{d}{2(2d+1)}m.
\]
\end{que}
\noindent One plausible step towards this question is to remove the term $\varepsilon n$ in Lemma~\ref{lem7}.
It may also be the case that such a bipartition can be replaced by a bisection.

\medskip

{\it E-mail address:} liuguanwu@ustc.edu.cn

\medskip
	
{\it E-mail address:} jiema@ustc.edu.cn

\medskip
	
{\it E-mail address:} zucle@mail.ustc.edu.cn

\end{document}